\documentclass[11pt]{amsart}
\usepackage{graphicx,amscd,color}
\usepackage{a4wide}

\theoremstyle{plain}
\newtheorem{theorem}{Theorem}[section]

\newtheorem{lemma}[theorem]{Lemma}

\theoremstyle{remark}
\newtheorem{remark}[theorem]{Remark}
\newtheorem{claim}{Claim}

\begin{document}

\title [Non-minimal bridge position of $2$-cable links]
{Non-minimal bridge position of $2$-cable links}

\author[J. H. Lee]{Jung Hoon Lee}
\address{Department of Mathematics and Institute of Pure and Applied Mathematics,
Jeonbuk National University, Jeonju 54896, Korea}
\email{junghoon@jbnu.ac.kr}


\begin{abstract}
Suppose that every non-minimal bridge position of a knot $K$ is perturbed.
We show that if $L$ is a $(2, 2q)$-cable link of $K$, then
every non-minimal bridge position of $L$ is also perturbed.
\end{abstract}

\maketitle

\section{Introduction}\label{sec1}

A knot in $S^3$ is said to be in {\em bridge position} with respect to a bridge sphere,
the original notion introduced by Schubert \cite{Schubert}, if
the knot intersects each of the $3$-balls bounded by the bridge sphere in a collection of $\partial$-parallel arcs.
It is generalized to knots (and links) in $3$-manifolds with the development of Heegaard splitting theory, and
is related to many interesting problems concerning, e.g. bridge number, (Hempel) distance, and incompressible surfaces in $3$-manifolds.

From any $n$-bridge position,
we can always get an $(n+1)$-bridge position by creating a new local minimum point and a nearby local maximum point of the knot.
A bridge position isotopic to one obtained in this way is said to be {\em perturbed}.
(It is said to be {\em stabilized} in some context.)
A bridge position of the unknot is unique in the sense that
any $n$-bridge ($n > 1$) position of the unknot is perturbed \cite{Otal1}.
The uniqueness also holds for $2$-bridge knots \cite{Otal2}.
See also \cite{Scharlemann-Tomova}, where all bridge surfaces for $2$-bridge knots are considered.
Ozawa \cite{Ozawa} showed that non-minimal bridge positions of torus knots are perturbed.
Zupan \cite{Zupan} showed such property for iterated torus knots and iterated cables of $2$-bridge knots.
More generally, he showed that if $K$ is an mp-small knot and every non-minimal bridge position of $K$ is perturbed, then
every non-minimal bridge position of a $(p,q)$-cable of $K$ is also perturbed \cite{Zupan}.
(Here, a knot is {\em mp-small} if its exterior contains no essential meridional planar surface.)
We remark that there exist examples of a knot with a non-minimal bridge position that is not perturbed \cite{Ozawa-Takao} and
furthermore, knots with arbitrarily high index bridge positions that are not perturbed \cite{JKOT}.

In this paper, we consider non-minimal bridge position of $2$-cable links of a knot $K$
without the assumption of mp-smallness of $K$.

\begin{theorem} \label{thm1}
Suppose that $K$ is a knot in $S^3$ such that every non-minimal bridge position of $K$ is perturbed.
Let $L$ be a $(2,2q)$-cable link of $K$.
Then every non-minimal bridge position of $L$ is also perturbed.
\end{theorem}

For the proof, we use the notion of t-incompressibility and t-$\partial$-incompressibility of \cite{Hayashi-Shimokawa}.
We isotope an annuls $A$ whose boundary is $L$ to a good position so that
it is t-incompressible and t-$\partial$-incompressible in one side, say $B_2$, of the bridge sphere.
Then $A \cap B_2$ consists of bridge disks and (possibly) properly embedded disks.
By using the idea of changing the order of t-$\partial$-compressions in \cite{Doll} or \cite{Hayashi-Shimokawa},
we show that in fact $A \cap B_2$ consists of bridge disks only.
Then by a further argument, we find a cancelling pair of disks for the bridge position.

\section{T-incompressible and t-$\partial$-incompressible surfaces in a $3$-ball}\label{sec2}

A {\em trivial tangle} $T$ is a union of properly embedded arcs $b_1, \ldots, b_n$ in a $3$-ball $B$ such that
each $b_i$ cobounds a disk $D_i$ with an arc $s_i$ in $\partial B$, and $D_i \cap (T - b_i) = \emptyset$.
By standard argument, $D_i$'s can be taken to be pairwise disjoint.

Let $F$ denote a surface in $B$ satisfying $F \cap (\partial B \cup T) = \partial F$.
A {\em t-compressing disk} for $F$ is a disk $D$ in $B - T$ such that $D \cap F = \partial D$ and
$\partial D$ is essential in $F$, i.e. $\partial D$ does not bound a disk in $F$.
A surface $F$ is {\em t-compressible} if there is a t-compressing disk for $F$, and
$F$ is {\em t-incompressible} if it is not t-compressible.

An arc $\alpha$ properly embedded in $F$ with its endpoints on $F \cap \partial B$, is {\em t-essential} if
$\alpha$ does not cobound a disk in $F$ with a subarc of $F \cap \partial B$.
In particular, an arc in $F$ parallel to a component of $T$ can be t-essential.
See Figure \ref{fig1}.
(Such an arc will be called {\em bridge-parallel} in Section \ref{sec3}.)
A {\em t-$\partial$-compressing disk} for $F$ is a disk $\Delta$ in $B - T$ such that
$\partial \Delta$ is an endpoint union of two arcs $\alpha$ and $\beta$, and
$\alpha = \Delta \cap F$ is t-essential, and $\beta = \Delta \cap \partial B$.
A surface $F$ is {\em t-$\partial$-compressible} if there is a t-$\partial$-compressing disk for $F$, and
$F$ is {\em t-$\partial$-incompressible} if it is not t-$\partial$-compressible.

\begin{figure}[!hbt]
\centering
\includegraphics[width=10cm,clip]{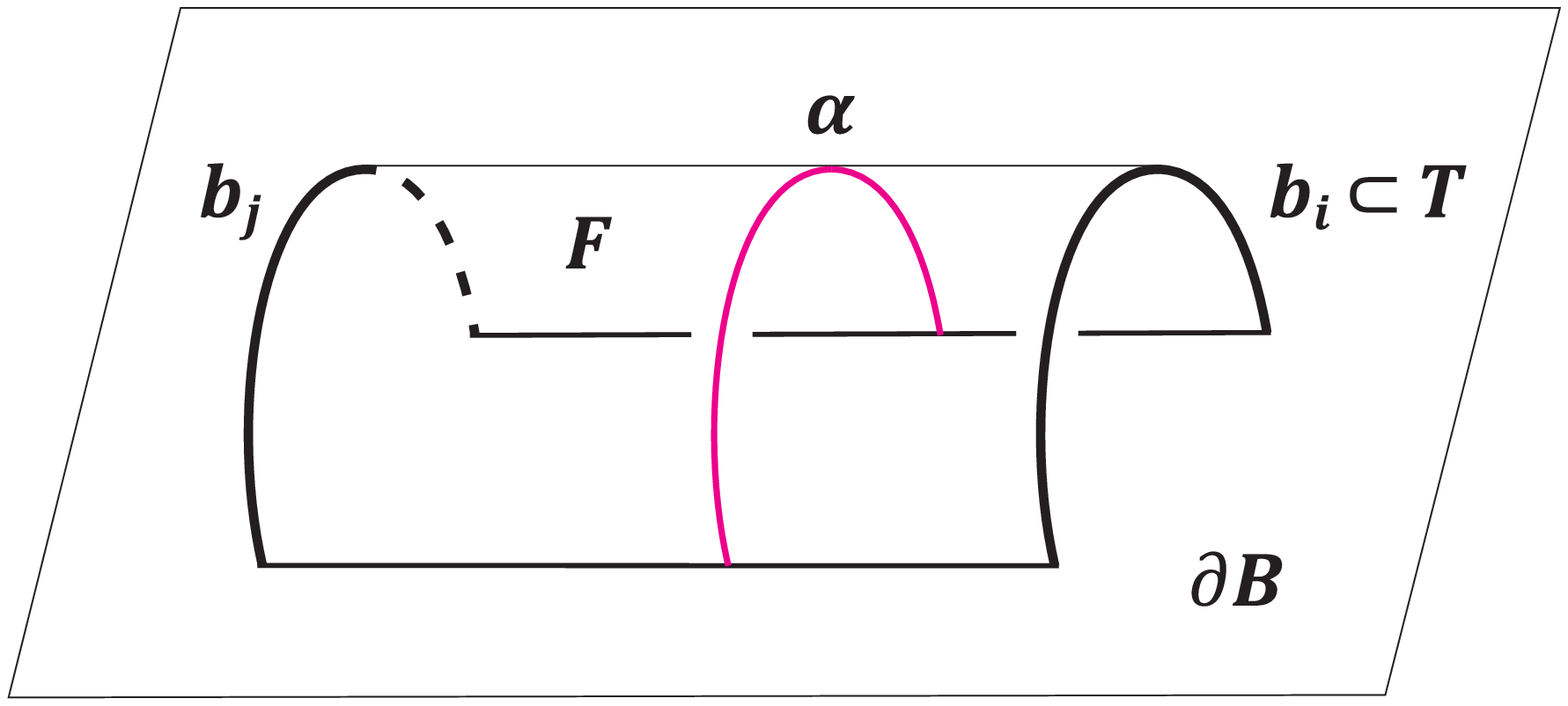}
\caption{A t-essential arc $\alpha$.}\label{fig1}
\end{figure}

Hayashi and Shimokawa \cite{Hayashi-Shimokawa} classified
t-incompressible and t-$\partial$-incompressible surfaces in a compression body in more general setting.
Here we give a simplified version of the theorem.

\begin{lemma}[\cite{Hayashi-Shimokawa}]\label{lem1}
Let $(B, T)$ be a pair of a $3$-ball and a trivial tangle in $B$, and
let $F \subset B$ be a surface satisfying $F \cap (\partial B \cup T) = \partial F \ne \emptyset$.
Suppose that $F$ is both t-incompressible and t-$\partial$-incompressible.
Then each component of $F$ is either
\begin{enumerate}
\item a disk $D_i$ cobounded by an arc $b_i$ of $T$ and an arc in $\partial B$ with $D_i \cap (T - b_i) = \emptyset$, or
\item a disk $C$ properly embedded in $B$ with $C \cap T = \emptyset$.
\end{enumerate}
\end{lemma}

\section{Bridge position}\label{sec3}

Let $S$ be a $2$-sphere decomposing $S^3$ into two $3$-balls $B_1$ and $B_2$.
Let $K$ denote a knot (or link) in $S^3$.
Then $K$ is said to be in {\em bridge position} with respect to $S$ if $K \cap B_i$ $(i=1,2)$ is a trivial tangle.
Each arc of $K \cap B_i$ is called a {\em bridge}.
If the number of bridges of $K \cap B_i$ is $n$, we say that $K$ is in {\em $n$-bridge position}.
The minimum such number $n$ among all bridge positions of $K$ is called the {\em bridge number} $b(K)$ of $K$.
A bridge cobounds a {\em bridge disk} with an arc in $S$, whose interior is disjoint from $K$.
We can take a collection of $n$ pairwise disjoint bridge disks by standard argument, and
it is called a {\em complete bridge disk system}.
For a bridge disk $D$ in, say $B_1$,
if there exists a bridge disk $E$ in $B_2$ such that $D \cap E$ is a single point of $K$,
then $D$ is called a {\em cancelling disk}.
We call $(D, E)$ a {\em cancelling pair}.

A {\em perturbation} is an operation on an $n$-bridge position of $K$
that creates a new local minimum and local maximum in a small neighborhood of a point of $K \cap S$,
resulting in an $(n+1)$-bridge position of $K$.
A bridge position obtained by a perturbation admits a cancelling pair by the construction.
Conversely, it is known that a bridge position admitting a cancelling pair
is isotopic to one obtained from a lower index bridge position by a perturbation.
(See e.g. \cite[Lemma 3.1]{Scharlemann-Tomova}).
Hence, as a definition, we say that a bridge position is {\em perturbed} if it admits a cancelling pair.

Let $V_1$ be a standard solid torus in $S^3$ with core $\alpha$, and
$V_2$ be a solid torus in $S^3$ whose core is a knot $C$.
A meridian $m_1$ of $V_1$ is uniquely determined up to isotopy.
Let $l_1 \subset \partial V_1$ be a longitude of $V_1$ such that
the linking number $\textrm{lk} (l_1, \alpha) = 0$, called the {\em preferred longitude}.
Similarly, let $m_2$ and $l_2$ be a meridian and a longitude of $V_2$ respectively such that $\textrm{lk} (l_2, C) = 0$.
Take a $(p,q)$-torus knot (or link) $T_{p,q}$ in $\partial V_1$ that wraps $V_1$ longitudinally $p$ times;
more precisely, $|T_{p,q} \cap m_1| = p$ and $|T_{p,q} \cap l_1| = q$.
Let $h: V_1 \to V_2$ be a homeomorphism sending $m_1$ to $m_2$ and $l_1$ to $l_2$.
Then $K = h(T_{p,q}) \subset S^3$ is called a {\em $(p,q)$-cable} of $C$.
Concerning the bridge number, it is known that $b(K) = p \cdot b(C)$ \cite{Schubert}, \cite{Schultens}.

\vspace{0.2cm}

Let $K$ be a knot (or link) in $n$-bridge position with respect to a decomposition $S^3 = B_1 \cup_S B_2$,
so $K \cap B_1$ is a union of bridges $b_1, \ldots, b_n$.
Let $\mathcal{R} = R_1 \cup \cdots \cup R_n$ be a complete bridge disk system for $\bigcup b_i$,
where $R_i$ is a bridge disk for $b_i$.
Let $F$ be a surface bounded by $K$ and $F_1 = F \cap B_1$.
When we move $K$ to an isotopic bridge position, $F_1$ moves together.
We consider $\mathcal{R} \cap F_1$.
By isotopy we assume that in a small neighborhood of $b_i$, $\mathcal{R} \cap F_1 = b_i$.

An arc $\gamma$ in $F_1$ is {\em bridge-parallel} ({\em b-parallel} briefly) if
$\gamma$ is parallel, in $F_1$, to some $b_i$ and cuts off a rectangle $P$ from $F_1$
whose four edges are $\gamma$, $b_i$, and two arcs in $S$.
Let $\alpha$($\ne b_k$) denote an arc of $\mathcal{R} \cap F_1$ which is outermost in some $R_k$ and
cuts off the corresponding outermost disk $\Delta$ disjoint from $b_k$.
The following lemma will be used in Section \ref{sec6}.

\begin{lemma}\label{lem2}
After possibly changing $K$ to an isotopic bridge position, there is no $\alpha$ that is b-parallel.
\end{lemma}

\begin{proof}
We isotope $K$ and $F_1$ and take $\mathcal{R}$ so that the minimal number of $| \mathcal{R} \cap F_1 |$ is realized.
Suppose that there is such an arc $\alpha$ which is parallel in $F_1$ to $b_i$ (same or not with $b_k$).
Isotope $b_i$ along $P$ to an arc parallel to $\alpha$ so that the changed surface $F'_1$ is disjoint from $\Delta$.
See Figure \ref{fig2}.
Take a new bridge disk $R'_i$ for $b_i$ to be a parallel copy of $\Delta$.
Other bridge disks $R_j$ ($j \ne i$) remain unaltered.
They are all mutually disjoint.
Hence $\mathcal{R}' = \mathcal{R} - R_i \cup R'_i$ is a new complete bridge disk system.
We see that $| \mathcal{R}' \cap F'_1 | < | \mathcal{R} \cap F_1 |$
since at least $\alpha$ no longer belongs to the intersection $\mathcal{R}' \cap F'_1$.
This contradicts the minimality of $| \mathcal{R} \cap F_1 |$.
\end{proof}

\begin{figure}[!hbt]
\centering
\includegraphics[width=8cm,clip]{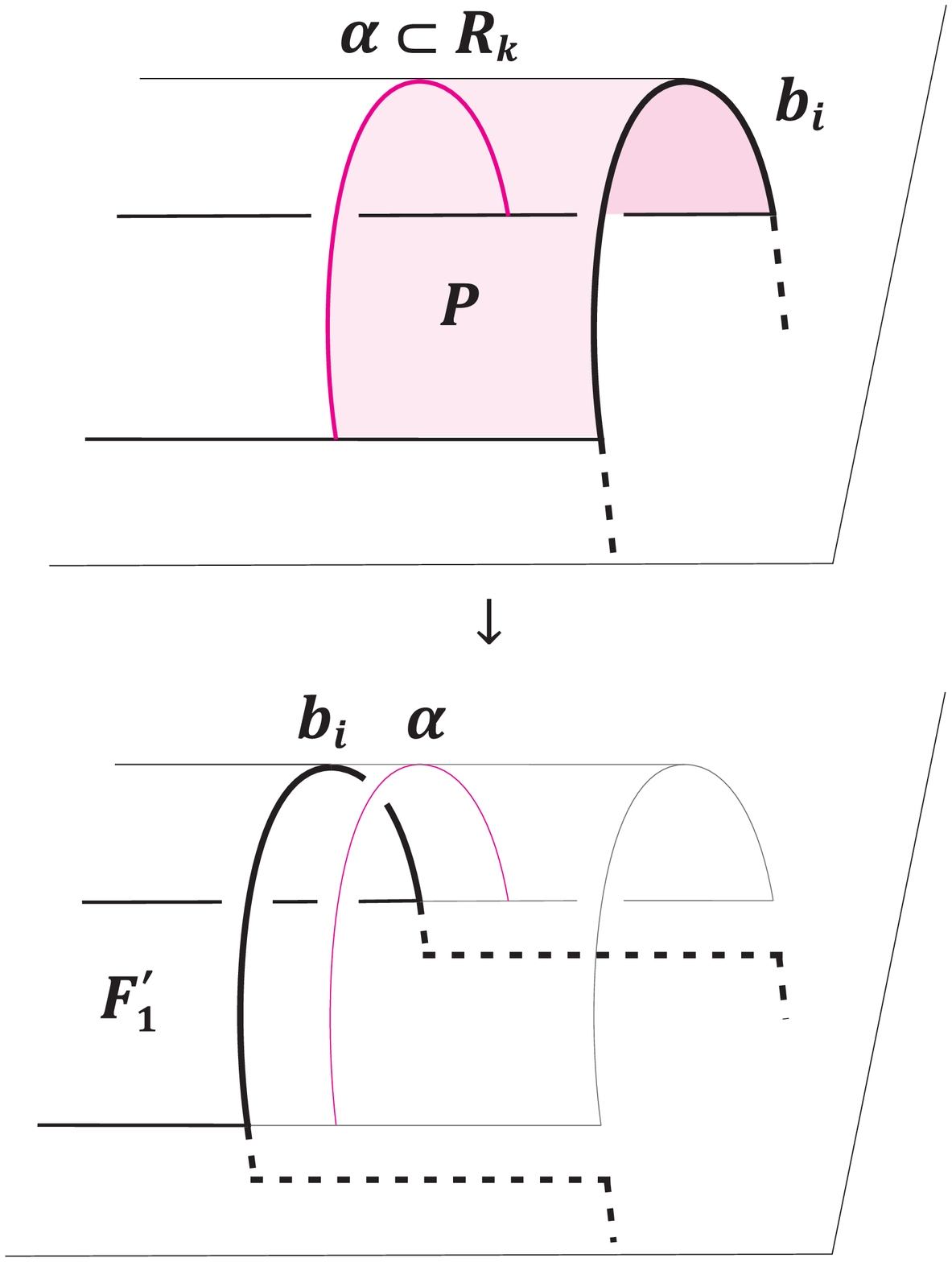}
\caption{Sliding $b_i$ along $P$.}\label{fig2}
\end{figure}

Now we consider a sufficient condition for a bridge position to be perturbed.

\begin{lemma}\label{lem3}
Suppose a separating arc $\gamma$ of $F \cap S$ cuts off a disk $\Gamma$ from $F$ such that
\begin{enumerate}
\item $\Gamma \cap B_1$ is a single disk $\Gamma_1$, and
\item $\Gamma \cap B_2$($\ne \emptyset$) consists of bridge disks $D_1, \ldots, D_k$.
\end{enumerate}
Then the bridge position of $K$ is perturbed.
\end{lemma}

\begin{proof}
Let $b_i$ ($i=1, \ldots, k$) denote the bridge for $D_i$ and $s_i = D_i \cap S$.
Let $r_1, \ldots, r_{k+1}$ denote the bridges contained in $\Gamma_1$.
We assume that $r_i$ is adjacent to $b_{i-1}$ and $b_i$.
See Figure \ref{fig3}.
Let $\mathcal{R} = R_1 \cup \cdots \cup R_{k+1}$ be a union of disjoint bridge disks, where
$R_i$ is a bridge disk for $r_i$.
In the following argument, we consider $\mathcal{R} \cap \Gamma_1$ except for $r_1 \cup \cdots \cup r_{k+1}$.

\begin{figure}[!hbt]
\centering
\includegraphics[width=14cm,clip]{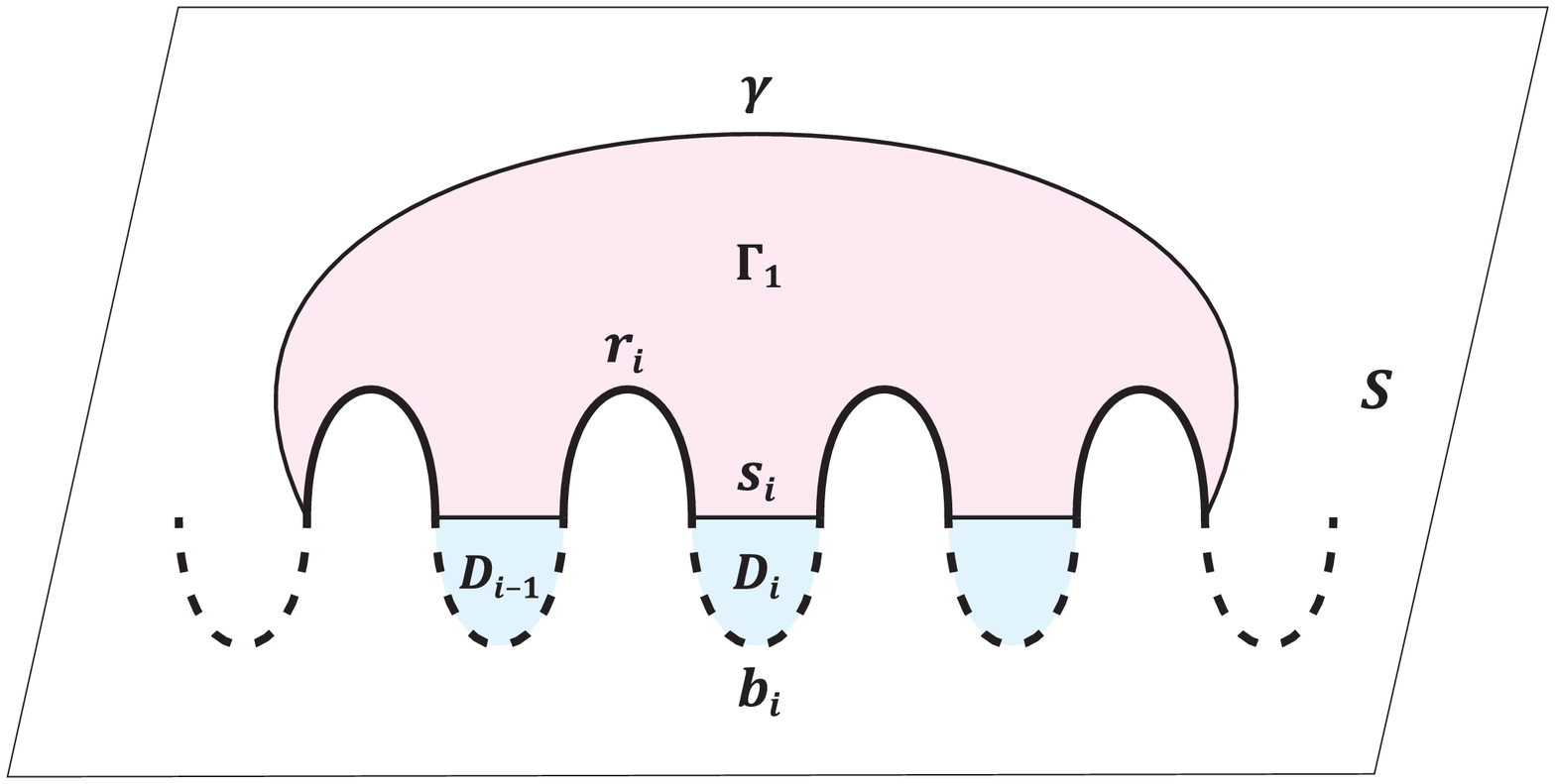}
\caption{The disk $\Gamma_1$ and bridge disks $D_i$'s.}\label{fig3}
\end{figure}

Suppose there is a circle component of $\mathcal{R} \cap \Gamma_1$.
Let $\alpha$($\subset R_i \cap \Gamma_1$) be one which is innermost in $\Gamma_1$ and
$\Delta$ be the innermost disk that $\alpha$ bounds.
Let $\Delta'$ be the disk that $\alpha$ bounds in $R_i$.
Then by replacing $\Delta'$ with $\Delta$, we can reduce $| \mathcal{R} \cap \Gamma_1 |$.
So we assume that there is no circle component of $\mathcal{R} \cap \Gamma_1$.

Suppose there is an arc component of $\mathcal{R} \cap \Gamma_1$ with both endpoints on the same arc of $\Gamma_1 \cap S$.
Let $\alpha$($\subset R_i \cap \Gamma_1$) be one which is outermost in $\Gamma_1$ and
$\Delta$ be the corresponding outermost disk in $\Gamma_1$ cut off by $\alpha$.
The arc $\alpha$ cuts $R_i$ into two disks and let $\Delta'$ be one of the two disks that does not contain $r_i$.
By replacing $\Delta'$ with $\Delta$, we can reduce $| \mathcal{R} \cap \Gamma_1 |$.
So we assume that there is no arc component of $\mathcal{R} \cap \Gamma_1$ with both endpoints on the same arc of $\Gamma_1 \cap S$.

If $\mathcal{R} \cap \Gamma_1 = \emptyset$, then $(R_i, D_i)$ is a desired cancelling pair and
the bridge position of $K$ is perturbed.
So we assume that $\mathcal{R} \cap \Gamma_1 \ne \emptyset$.

\vspace{0.2cm}

Case $1$. Every arc of $\mathcal{R} \cap \Gamma_1$ is b-parallel.

Let $\alpha$ be an arc of $\mathcal{R} \cap \Gamma_1$ which is outermost in some $R_j$ and
$\Delta$ be the outermost disk that $\alpha$ cuts off from $R_j$.
In addition, let $\alpha$ be b-parallel to $r_i$ via a rectangle $P$.
Then $P \cup \Delta$ is a new bridge disk for $r_i$, and
$(P \cup \Delta, D_i)$ or $(P \cup \Delta, D_{i-1})$ is a cancelling pair.

\vspace{0.2cm}

Case $2$. There is a non-b-parallel arc of $\mathcal{R} \cap \Gamma_1$.

Consider only non-b-parallel arcs of $\mathcal{R} \cap \Gamma_1$.
Let $\beta$ denote one which is outermost in $\Gamma_1$ among them and
$\Gamma_0$ denote the outermost disk cut off by $\beta$.
Because there are at least two outermost disks, we take $\Gamma_0$ such that $\partial \Gamma_0$ contains some $s_i$.
Let $r_l, \ldots, r_m$ be the bridges contained in $\Gamma_0$ and
$\mathcal{R}' = R_l \cup \cdots \cup R_m$.
In the following, we consider $\mathcal{R}' \cap \Gamma_0$ except for $r_l \cup \cdots \cup r_m$ and $\beta$.
If $\mathcal{R}' \cap \Gamma_0 = \emptyset$, then there exists a cancelling pair $(R_i, D_i)$.
Otherwise, every arc of $\mathcal{R}' \cap \Gamma_0$ is b-parallel.
Let $\alpha$ be an arc of $\mathcal{R}' \cap \Gamma_0$ which is outermost in some $R_j$ and
$\Delta$ be the outermost disk that $\alpha$ cuts off from $R_j$.
In addition, let $\alpha$ be b-parallel to $r_i$ via a rectangle $P$.
Then $P \cup \Delta$ is a new bridge disk for $r_i$, and
$(P \cup \Delta, D_i)$ or $(P \cup \Delta, D_{i-1})$ is a cancelling pair.
\end{proof}

\begin{figure}[!hbt]
\centering
\includegraphics[width=10.5cm,clip]{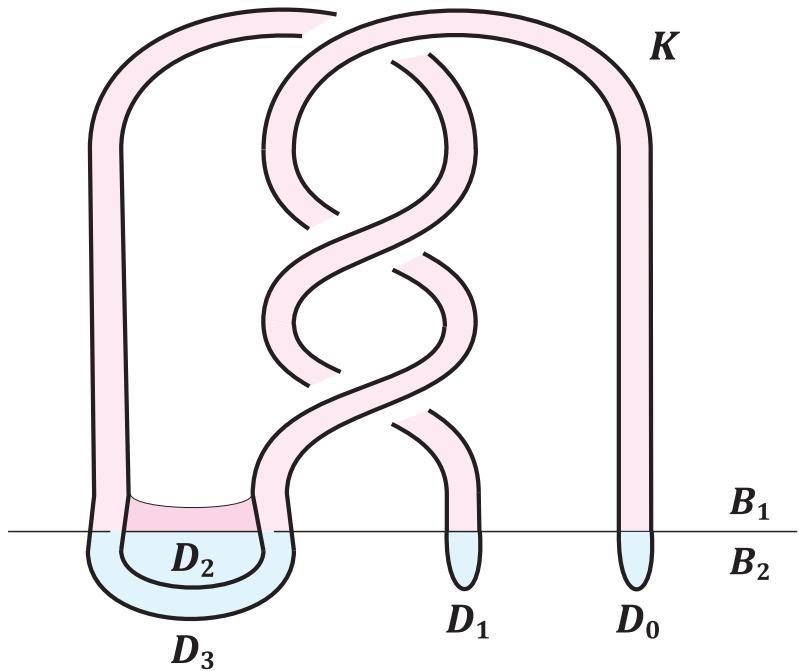}
\caption{The disks $D_0$ and $D_1$ are cancelling disks, whereas $D_2$ and $D_3$ are not.}\label{fig4}
\end{figure}

\begin{remark}
Let $K$ be an unknot in $n$-bridge position,
with $K \cap B_1 = r_0 \cup r_1 \cup \cdots \cup r_{n-1}$ and $K \cap B_2 = b_0 \cup b_1 \cup \cdots \cup b_{n-1}$.
We assume that the bridge $r_i$ ($i=0,1,\ldots,n-1$) is adjacent to $b_{i-1}$ and $b_i$, where
we consider the index $i$ modulo $n$.
Let $R_i$ and $D_i$ denote bridge disks for $r_i$ and $b_i$ respectively.
Then $\mathcal{C} = R_0 \cup D_0 \cup \cdots \cup R_{n-1} \cup D_{n-1}$
is called a {\em complete cancelling disk system} if
each $(R_i, D_{i-1})$ and each $(R_i, D_i)$ is a cancelling pair.

Let $D$ denote a disk bounded by $K$.
By following the argument of the proof of Theorem \ref{thm1},
we can assume that $D \cap B_2$ consists of bridge disks $D_0, \ldots, D_{n-1}$ and
$D \cap B_1$ is a single disk, as in \cite{Hayashi-Shimokawa}.
Then if $n > 1$, the bridge position of $K$ admits a cancelling pair by Lemma \ref{lem3},
giving a proof of the uniqueness of bridge position of the unknot.
One may hope that $D_0 \cup \cdots \cup D_{n-1}$ extends to a complete cancelling disk system.
But when $n \ge 4$, there exists an example such that
$D_0 \cup \cdots \cup D_{n-1}$ does not extend to a complete cancelling disk system,
as expected in \cite[Remark 1.2]{Hayashi-Shimokawa}.
Some $D_i$ is even not a cancelling disk.
This issue was related to one of the motivations for the present work.
In Figure \ref{fig4}, $K$ is an unknot in $4$-bridge position bounding a disk $D$ and
$D \cap B_2 = D_0 \cup D_1 \cup D_2 \cup D_3$.
Each of the disks $D_0$ and $D_1$ is a cancelling disk.
However, $D_2$ and $D_3$ are not cancelling disks because, say for $D_2$,
an isotopy of $b_2$ along $D_2$ and then slightly into $B_1$ does not give a $3$-bridge position of $K$
(see \cite{Scharlemann-Tomova}, \cite{Lee}).
\end{remark}

\section{Proof of Theorem \ref{thm1}: First step} \label{sec4}

Let $K$ be a knot such that every non-minimal bridge position of $K$ is perturbed.
Let $L$ be a $(2,2q)$-cable link of $K$, with components $K_1$ and $K_2$.
Suppose that $L$ is in non-minimal bridge position with respect to a bridge sphere $S$ bounding $3$-balls $B_1$ and $B_2$.
Each $L \cap B_i$ $(i=1,2)$ is a trivial tangle.
Since $L$ is a $2$-cable link, $L$ bounds an annulus, denoted by $A$.
We take $A$ so that $|A \cap S|$ is minimal.

\begin{claim}\label{claim1}
One of the following holds.
\begin{itemize}
\item $L$ is the unlink in a non-minimal bridge position, hence perturbed.
\item $A \cap B_2$ is t-incompressible in $B_2$.
\end{itemize}
\end{claim}

\begin{proof}
Suppose that $A \cap B_2$ is t-compressible.
Let $\Delta$ be a t-compressing disk for $A \cap B_2$ and let $\alpha = \partial \Delta$.
Let $F$ be the component of $A \cap B_2$ containing $\alpha$.

Case $1$. $\alpha$ is essential in $A$.

A t-compression of $A$ along $\Delta$ gives two disjoint disks bounded by $K_1$ and $K_2$ respectively.
Then $L$ is an unlink.
Since the complement of an unlink has a reducing sphere,
by \cite{Bachman-Schleimer} a bridge position of an unlink is a split union of bridge positions of unknot components.
Since a non-minimal bridge position of the unknot is perturbed, we see that $L$ is perturbed.

Case $2$. $\alpha$ is inessential in $A$.

Let $\Delta '$ be the disk that $\alpha$ bounds in $A$.
Then $(\textrm{Int} \, \Delta ') \cap S \ne \emptyset$, since otherwise $\alpha$ is inessential in $F$.
By replacing $\Delta '$ of $A$ with $\Delta$, we get a new annulus $A'$ bounded by $L$ such that
$| A' \cap S | < | A \cap S |$, contrary to the minimality of $| A \cap S |$.
\end{proof}

Since our goal is to show that the bridge position of $L$ is perturbed, from now on we assume that $L$ is not the unlink.
By Claim \ref{claim1}, $A \cap B_2$ is t-incompressible in $B_2$.
If $A \cap B_2$ is t-$\partial$-compressible in $B_2$, we do a t-$\partial$-compression.

\begin{claim}\label{claim2}
A t-$\partial$-compression preserves the t-incompressibility of $A \cap B_2$.
\end{claim}

\begin{proof}
Let $\Delta$ be a t-$\partial$-compressing disk for $A \cap B_2$.
Suppose that the surface after the t-$\partial$-compression along $\Delta$ is t-compressible.
A t-compressing disk $D$ can be isotoped to be disjoint from two copies of $\Delta$ and the product region $\Delta \times I$.
Then $D$ would be a t-compressing disk for $A \cap B_2$ before the t-$\partial$-compression, a contradiction.
\end{proof}

A t-$\partial$-compression simplifies a surface because it cuts the surface along a t-essential arc.
So if we maximally t-$\partial$-compress $A \cap B_2$, we obtain a t-$\partial$-incompressible $A \cap B_2$.
Note that the effect on $A$ of a t-$\partial$-compression of $A \cap B_2$ is
just pushing a neighborhood of an arc in $A$ into $B_1$,
which is called {\em an isotopy of Type $A$} in \cite{Jaco}.
After a maximal sequence of t-$\partial$-compressions,
$A \cap B_2$ is both t-incompressible and t-$\partial$-incompressible by Claim \ref{claim2}.
Then by applying Lemma \ref{lem1},

\begin{itemize}
\item[($\ast$)] $A \cap B_2$ consists of bridge disks $D_i$'s and properly embedded disks $C_j$'s.
\end{itemize}


\section{Proof of Theorem \ref{thm1}: T-$\partial$-compression and its dual operation} \label{sec5}

Take an annulus $A$ bounded by $L$ so that $(\ast )$ holds and the number $m$ of properly embedded disks $C_j$ is minimal.
In this section, we will show that $m=0$, i.e. $A \cap B_2$ consists of bridge disks only.

Suppose that $m > 0$.
Then $A \cap B_1$ is homeomorphic to an $m$-punctured annulus.
A similar argument as in the proof of Claim \ref{claim1} leads to that $A \cap B_1$ is t-incompressible.
By Lemma \ref{lem1} again, $A \cap B_1$ is t-$\partial$-compressible.
We can do a sequence of t-$\partial$-compressions on $A \cap B_1$ until it becomes t-$\partial$-incompressible.
Note that the t-incompressibility of $A \cap B_1$ is preserved.

Now we are going to define a t-$\partial$-compressing disk $\Delta_i$ ($i=0,1, \ldots, s$ for some $s$) and
its {\em dual disk} $U_{i+1}$ inductively.
Let $A_0 = A$.
Let $\Delta_0$ be a t-$\partial$-compressing disk for $A_0 \cap B_1$ and
$\alpha_0 = \Delta_0 \cap A_0$ and $\beta_0 = \Delta_0 \cap S$.
By a t-$\partial$-compression along $\Delta_0$,
a neighborhood of $\alpha_0$ is pushed along $\Delta_0$ into $B_2$ and thus a band $b_1$ is created in $B_2$.
Let $A_1$ denote the resulting annulus bounded by $L$.
Let $U_1$ be a dual disk for $\Delta_0$, that is,
a disk such that an isotopy of Type A along $U_1$ recovers a surface isotopic to $A_0$.
For the next step, let $\mathcal{U}_1 = U_1$.

Let $\Delta_1$ be a t-$\partial$-compressing disk for $A_1 \cap B_1$ and
$\alpha_1 = \Delta_1 \cap A_1$ and $\beta_1 = \Delta_1 \cap S$.
After a t-$\partial$-compression along $\Delta_1$, a band $b_2$ is created in $B_2$.
Let $A_2$ denote the resulting annulus bounded by $L$.
There are three cases to consider.

Case $1$. $\beta_1$ intersects the arc $\mathcal{U}_1 \cap S$ more than once.

The band $b_2$ cuts off small disks $U_{2,1}, U_{2,2}, \ldots, U_{2,k_2}$ from $\mathcal{U}_1$,
which are mutually parallel along the band.
We designate any one among the small disks, say $U_{2,1}$, as the dual disk $U_2$.
Let $\mathcal{R}_2 = \bigcup^{k_2}_{j=2} U_{2,j}$ be the union of others.

Case $2$. $\beta_1$ intersects $\mathcal{U}_1 \cap S$ once.

We take the subdisk that $b_2$ cuts off from $\mathcal{U}_1$ as the dual disk $U_2$, and
let $\mathcal{R}_2 = \emptyset$ in this case.

Case $3$. $\beta_1$ does not intersect $\mathcal{U}_1 \cap S$.

We take a dual disk $U_2$ freely, and let $\mathcal{R}_2 = \emptyset$ in this case.

In any case, let $\mathcal{U}_2 = \mathcal{U}_1 \cup U_2 - \mathrm{int}\, \mathcal{R}_2$.

\vspace{0.2cm}

In general, assume that $A_i$ and $\mathcal{U}_i$ are defined.
Let $\Delta_i$ be a t-$\partial$-compressing disk for $A_i \cap B_1$ and
$\alpha_i = \Delta_i \cap A_i$ and $\beta_i = \Delta_i \cap S$.
After a t-$\partial$-compression along $\Delta_i$, a band $b_{i+1}$ is created in $B_2$.
Let $A_{i+1}$ denote the resulting annulus bounded by $L$.

{\bf Case a}. $\beta_i$ intersects the collection of arcs $\mathcal{U}_i \cap S$ more than once.

The band $b_{i+1}$ cuts off small disks $U_{i+1,1}, U_{i+1,2}, \ldots, U_{i+1,k_{i+1}}$ from $\mathcal{U}_i$,
which are mutually parallel along the band.
We designate any one among the small disks, say $U_{i+1,1}$, as the dual disk $U_{i+1}$.
Let $\mathcal{R}_{i+1} = \bigcup^{k_{i+1}}_{j=2} U_{i+1,j}$ be the union of others.

{\bf Case b}. $\beta_i$ intersects $\mathcal{U}_i \cap S$ once.

We take the subdisk that $b_{i+1}$ cuts off from $\mathcal{U}_i$ as the dual disk $U_{i+1}$, and
let $\mathcal{R}_{i+1} = \emptyset$ in this case.

{\bf Case c}. $\beta_i$ does not intersect $\mathcal{U}_i \cap S$.

We take a dual disk $U_{i+1}$ freely, and let $\mathcal{R}_{i+1} = \emptyset$ in this case.

In any case, let $\mathcal{U}_{i+1} = \mathcal{U}_i \cup U_{i+1} - \mathrm{int}\, \mathcal{R}_{i+1}$.

\vspace{0.2cm}

Later, we do isotopy of Type A, dual to the t-$\partial$-compression, in reverse order along $U_{i+1}, U_i, \ldots, U_1$.
Let us call it {\em dual operation} for our convenience..
Let $b_{i+1}$ be the band mentioned above,
cutting off $U_{i+1,1}, U_{i+1,2}, \ldots, U_{i+1,k_{i+1}}$ from $\mathcal{U}_i$ (in {\bf Case a}).
When the dual operation along $U_{i+1}$ is done,
we modify every $U_j$ and $\mathcal{U}_j$ ($j \le i$) containing any $U_{i+1, s}$ ($s > 1)$ of $\mathcal{R}_{i+1}$,
by replacing each $U_{i+1,s}$ ($s > 1$) with
the union of a subband of $b_{i+1}$ between $U_{i+1,s}$ and $U_{i+1}$($= U_{i+1,1}$) and a copy of $U_{i+1}$, and
doing a slight isotopy.
We remark that, although it is not illustrated in Figure \ref{fig5},
some $U_j$'s and $\mathcal{U}_j$'s temporarily become immersed
when the subband passes through some removed region, say $U_{r,s}$ ($s > 1$).
But the $U_{r,s}$ ($s > 1$) is also modified as we proceed the dual operations, and
the $U_j$'s and $\mathcal{U}_j$'s again become embedded.
(In Figure \ref{fig5} and Figure \ref{fig6}, the dual operation along $U_{i+1}$ is done, and
the dual operation along $U_i$ is not done yet.)
Actually, before the sequence of dual operations,
$U_j$'s and $\mathcal{U}_j$'s ($j \le k$) are modified in advance so that
$\mathcal{U}_k$ is disjoint from the union of certain band $b_{k+1}$ and a disk $C_l$
(which will be explained later).
See Figure \ref{fig6}.

\begin{figure}[!hbt]
\centering
\includegraphics[width=13cm,clip]{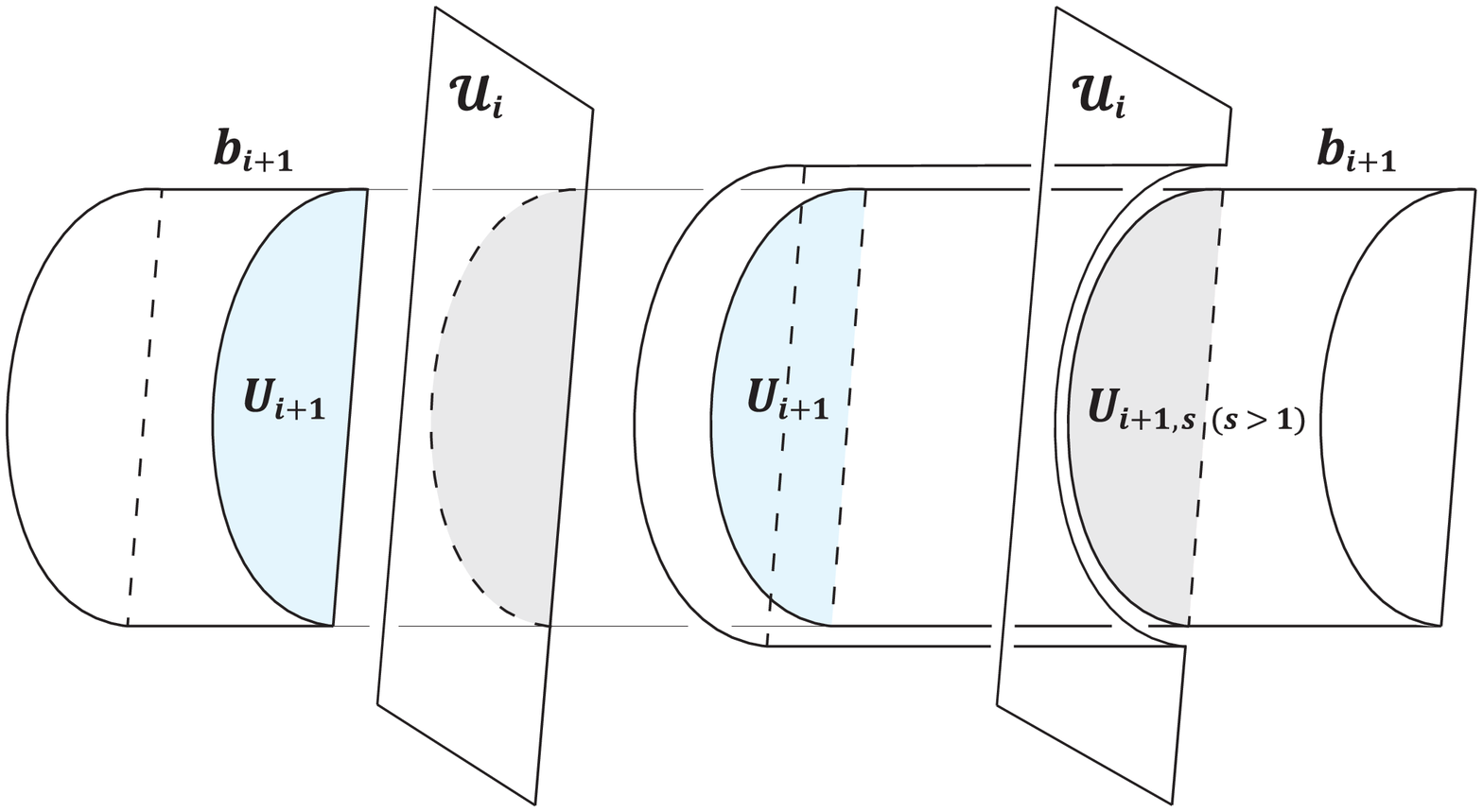}
\caption{Modifying $U_j$'s and $\mathcal{U}_j$'s ($j \le i$) containing any $U_{i+1,s}$ ($s > 1$),
after the dual operation along $U_{i+1}$ and before the dual operation along $U_i$.}\label{fig5}
\end{figure}

\begin{figure}[!hbt]
\centering
\includegraphics[width=14.9cm,clip]{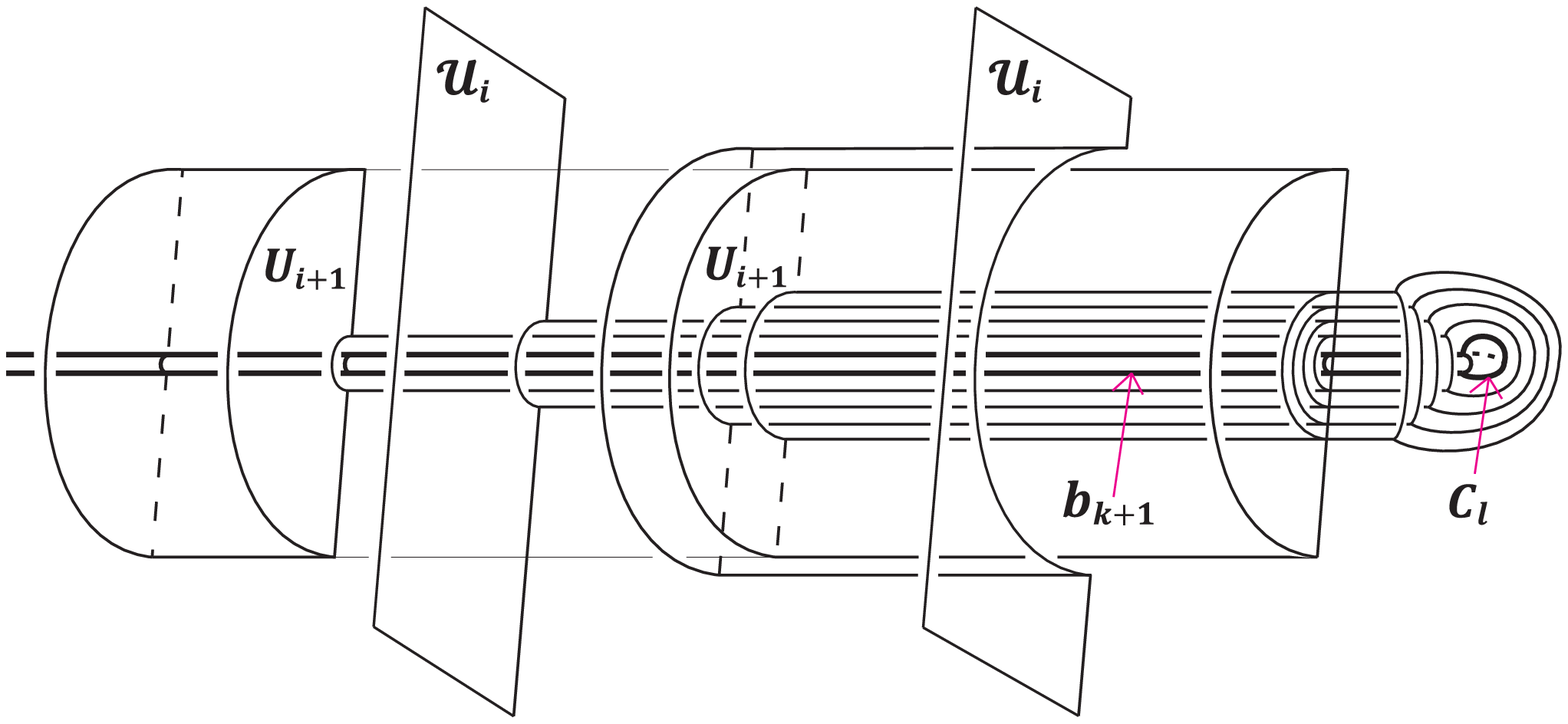}
\caption{Making $\mathcal{U}_k$ to be disjoint from $b_{k+1} \cup C_l$.}\label{fig6}
\end{figure}

The t-$\partial$-compressing disk $\Delta_i$ is taken to be disjoint from two copies of $\Delta_j$ ($j < i$).
Moreover, for every $i$ we can draw $\alpha_i$ on $A$($= A_0$).
Since the t-$\partial$-compression along $\Delta_i$ is equivalent to cutting along $\alpha_i$ and
the sequence $\Delta_0, \Delta_1, \ldots, \Delta_s$ is maximal,
every $c_j = \partial C_j$ is incident to some $\alpha_i$.

\begin{claim}\label{claim3}
For each $c_j$, there exists an $\alpha_i$ such that $\alpha_i$ connects $c_j$ to other component of $A \cap S$.
\end{claim}

\begin{proof}
Suppose that there exists a $c_j$ which is not connected to other component of $A \cap S$.
That is, for such $c_j$, every $\alpha_i$ incident to $c_j$ connects $c_j$ to itself.
Then after a maximal sequence of t-$\partial$-compressions on $A$, some non-disk components will remain.
This contradicts Lemma \ref{lem1}.
\end{proof}

Let $k$ be the smallest index such that $\alpha_k$ connects some $c_j$, say $c_l$,
to other component (other $c_j$ or $D_i \cap S$).
If $k = 0$, then by a t-$\partial$-compression along $\Delta_0$,
either $C_l$ and other $C_j$ are merged into one properly embedded disk, or
$C_l$ and a bridge disk are merged into a new bridge disk.
This contradicts the minimality of $m$.
So we assume that $k \ge 1$.

Suppose that we performed t-$\partial$-compressions along $\Delta_0, \Delta_1, \ldots, \Delta_k$.
Consider the small disks that the band $b_{k+1}$ cuts off from $\mathcal{U}_k$.
They are parallel along $b_{k+1}$.
We replace the small disks one by one, the nearest one to $C_l$ first, so that
$\mathcal{U}_k$ is disjoint from $b_{k+1} \cup C_l$.
Let $\Delta$ be the small disk nearest to $C_l$.
Let $\Delta '$ be the union of a subband of $b_{k+1}$ and $C_l$
that $\Delta \cap b_{k+1}$ cuts off from $b_{k+1} \cup C_l$.
For every $U_j$ and $\mathcal{U}_j$ ($j \le k$) containing $\Delta$, we replace $\Delta$ with $\Delta '$.
Then again let $\Delta$ be the (next) small disk nearest to $C_l$ and we repeat the above operation until
$\mathcal{U}_k$ is disjoint from $b_{k+1} \cup C_l$.

Now we do the dual operation on $A_k$ in reverse order along $U_k, U_{k-1}, \ldots, U_1$.
Let $A'_{i-1}$ ($i=1, \ldots, k$) be the resulting annulus after the dual operation along $U_i$.
The shape of the dual disk $U_k$ is possibly changed but
the number of circle components and arc components of $A'_{k-1} \cap S$ is
same with those of $A_{k-1} \cap S$.
After the dual operation along $U_k$,
it is necessary to modify some $U_j$'s and $\mathcal{U}_j$'s ($j \le k-1$) further as in Figure \ref{fig6}
so that $U_{k-1}$ is disjoint from $b_{k+1} \cup C_l$.
In this way, we do the sequence of dual operations, and
the number of circle components of $A'_0 \cap S$ is also $m$.
Then because $b_{k+1}$ is disjoint from $U_1, \ldots, U_k$,
we can do the t-$\partial$-compression of $A'_0$ along $\Delta_k$ first and
the number $m$ of properly embedded disks $C_j$ is reduced, contrary to our assumption.

We have shown the following claim.

\begin{claim}\label{claim4}
$A \cap B_2$ consists of bridge disks.
\end{claim}

\section{Proof of Theorem \ref{thm1}: Finding a cancelling pair} \label{sec6}

By Claim \ref{claim4}, $A \cap B_2$ consists of bridge disks.
Let $d_0, \ldots, d_{k-1}$ and $e_0, \ldots, e_{l-1}$ be bridges of $K_1 \cap B_2$ and $K_2 \cap B_2$ respectively
that are indexed consecutively along each component.
Let $D_i \subset A \cap B_2$ ($i = 0, \ldots, k-1$) be the bridge disk for $d_i$ and $s_i = D_i \cap S$, and
let $E_j \subset A \cap B_2$ ($j = 0, \ldots, l-1$) be the bridge disk for $e_j$ and $t_j = E_j \cap S$.
Let $\mathcal{R} = R_1 \cup \cdots \cup R_{k+l}$ be a complete bridge disk system for $L \cap B_1$ and
let $F = A \cap B_1$.
We consider $\mathcal{R} \cap F$ except for the bridges $L \cap B_1$.

If there is an inessential circle component of $\mathcal{R} \cap F$ in $F$,
it can be removed by standard innermost disk argument.
If there is an essential circle component of $\mathcal{R} \cap F$ in $F$,
then $L$ would be the unlink as in Case $1$ of the proof of Claim \ref{claim1}.
So we assume that there is no circle component of $\mathcal{R} \cap F$.
If there is an inessential arc component of $\mathcal{R} \cap F$ in $F$
with both endpoints on the same $s_i$ (or $t_j$),
then the arc can be removed by standard outermost disk argument.
So we assume that there is no arc component of $\mathcal{R} \cap F$ with both endpoints on the same $s_i$ (or $t_j$).

If $\mathcal{R} \cap F = \emptyset$, we easily get a cancelling pair, say $(R_m, D_i)$ or $(R_m, E_j)$,
so we assume that $\mathcal{R} \cap F \ne \emptyset$.
Let $\alpha$ denote an arc of $\mathcal{R} \cap F$ which is outermost in some $R_m$ and
let $\Delta$ denote the outermost disk that $\alpha$ cuts off from $R_m$.
Applying Lemma \ref{lem2}, $\alpha$ is not b-parallel.
Suppose that one endpoint of $\alpha$ is in, say $s_{i_1}$, and the other is in $s_{i_2}$
with the cyclic distance $d(i_1, i_2) = \textrm{min} \{ |i_1 - i_2|, |k - (i_1 - i_2)| \}$ greater than $1$.
Then after the t-$\partial$-compression along $\Delta$,
we get a subdisk of $A$ satisfying the assumption of Lemma \ref{lem3}, hence $L$ is perturbed.
So without loss of generality, we assume that one endpoint of $\alpha$ is in $s_0$ and the other is in $t_0$.

Let $A_1$ be the annulus obtained from $A$ by the t-$\partial$-compression along $\Delta$ and $F_1 = A_1 \cap B_1$.
The bridge disks $D_0$ and $E_0$ are connected by a band, and
let $P_1$ be the resulting rectangle with four edges $d_0, e_0$, and two arcs in $S$, say $p_1, p_2$.
Let $\alpha_1$ denote an arc of $\mathcal{R} \cap F_1$ which is outermost in some $R_m$ and
let $\Delta_1$ denote the outermost disk cut off by $\alpha_1$.
If at least one endpoint of $\alpha_1$ is contained in $p_1$ or $p_2$,
or one endpoint of $\alpha_1$ is in $s_{i_1}$($t_{j_1}$ respectively) and
the other is in $s_{i_2}$($t_{j_2}$ respectively),
then similarly as above,
\begin{itemize}
\item either $\alpha_1$ is inessential with both endpoints on the same component of $F_1 \cap S$, or
\item $\alpha_1$ is b-parallel, or
\item Lemma \ref{lem3} can be applied.
\end{itemize}
Hence we may assume that one endpoint of $\alpha_1$ is in $s_i$ ($i \ne 0$) and the other is in $t_j$ ($j \ne 0$).
After the t-$\partial$-compression along $\Delta_1$, $D_i$ and $E_j$ are merged into a rectangle.
Arguing in this way, each $D_i$ ($i=0, \ldots, k-1$) is merged with some $E_j$ because of the fact that
$\mathcal{R} \cap s_i = \emptyset$ gives us a cancelling pair.
Moreover, we see that $k = l$.
After $k$ successive t-$\partial$-compressions on $A$,
the new annulus $A'$ intersects $B_1$ and $B_2$ alternately, in rectangles.

Note that $b(L) = 2 b(K)$ and $L$ is in $2k$-bridge position.
Since $L$ is in non-minimal bridge position, $k > b(K)$.
So by the assumption of the theorem, the bridge position of $K_i$ ($i=1,2$) is perturbed.
Let $(D, E)$ be a cancelling pair for $K_1$ with $D \subset B_1$ and $E \subset B_2$.
However, $D$ and $E$ may intersect $K_2$.
Let $P_i$ and $P_{i+1}$ be any adjacent rectangles of $A'$ in $B_1$ and $B_2$ respectively.
We remove any unnecessary intersection of $D \cap P_i$ and $E \cap P_{i+1}$, the nearest one to $K_2$ first,
by isotopies along subdisks of $P_i$ and $P_{i+1}$ respectively.
See Figure $7$ for an example.
Then $(D, E)$ becomes a cancelling pair for the bridge position of $L$ as desired.

\begin{figure}[!hbt]
\centering
\includegraphics[width=13cm,clip]{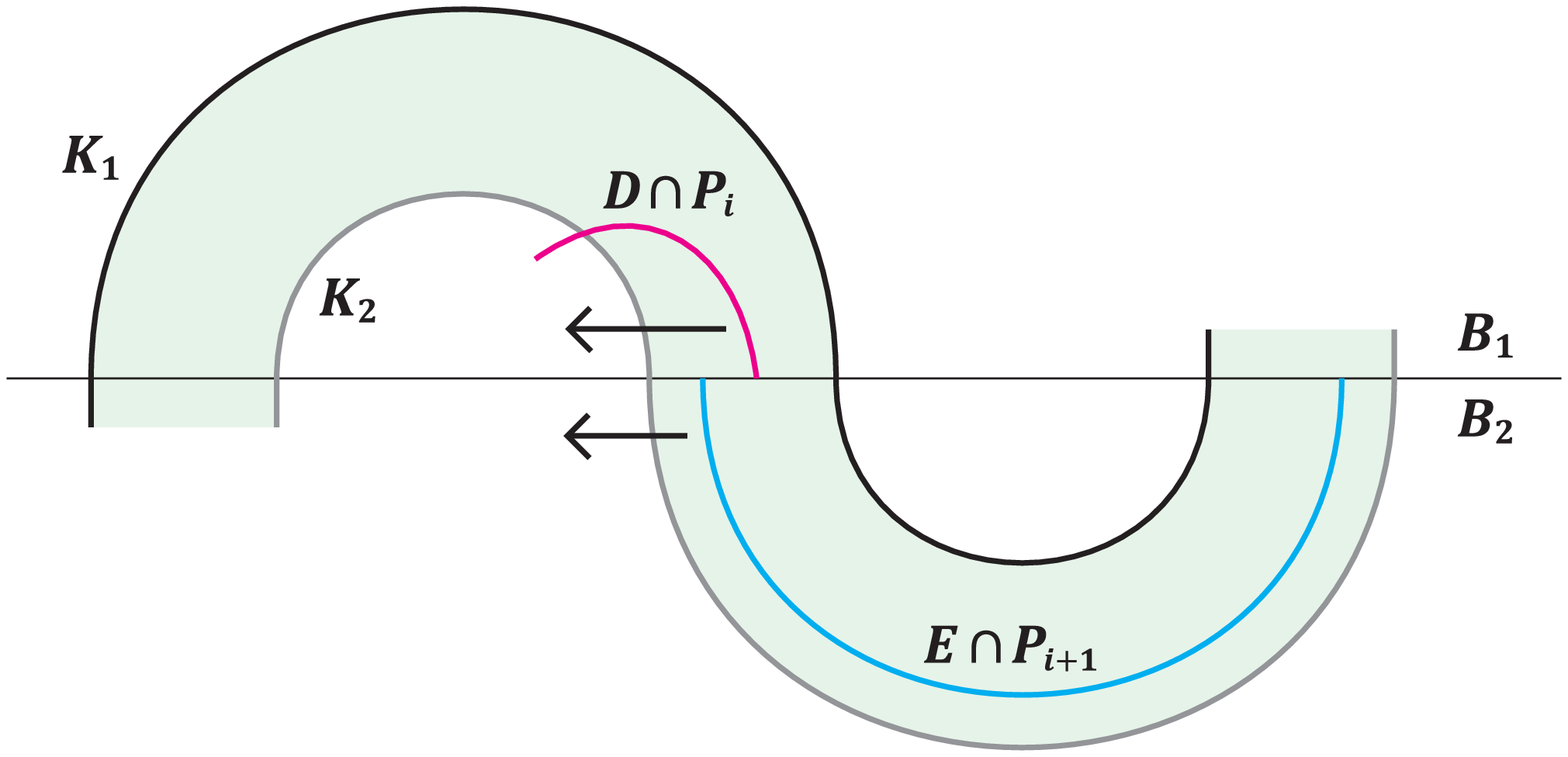}
\caption{Isotoping $D$ and $E$.}\label{fig7}
\end{figure}

\vspace{0.2cm}

{\noindent \bf Acknowledgments.}

The author was supported by the Basic Science Research Program through the National Research Foundation of Korea (NRF)
funded by the Ministry of Education (2018R1D1A1A09081849).

\end{document}